\documentclass[12pt]{amsart}
\usepackage{a4wide, amssymb, bbm, calligra, mathrsfs}
\usepackage[2cell,arrow,curve,matrix]{xy}
\UseHalfTwocells \UseTwocells

\newtheorem{Pro}{Proposition}[subsection]
\newtheorem{Le}[Pro]{Lemma}
\newtheorem{Th}[Pro]{Theorem}
\newtheorem{Co}[Pro]{Corollary}
\theoremstyle{definition}

\theoremstyle{remark}
\newtheorem{Exm}[Pro]{Example}

\def\fork{\pitchfork}

\def\D{{\mathscr D}}

\def\m{{\mathfrak m}}
\def\n{{\mathfrak n}}

\def\af{\mathfrak a}
\def\fb{\mathfrak b}
\def\cf{\mathfrak c}

\def\pf{\mathfrak p}
\def\qf{\mathfrak q}

\def\fct{{\bf Fct}}

\def\fa{\mathfrak F}

\def\ea{{\mathscr T}}

\def\fa{{\mathscr F}}

\let\al\alpha
\let\bb\beta

\let\oO\Omega

\let\xto\xrightarrow

\def\0{^{[1]}}

\def\op{^{\mathrm{op}}}
\def\1{^{-1}}

\def\cref#1#2#3{\left(#2\right.\left|\ #3\right)_{#1}}

\def\ta{{\mathscr T}}

\def\ka{{\mathscr K}}
\def\f{{\mathscr F}}

\def\Topo{{\bf Top}}

\def\ope{{\bf Open}}

\def\vU{{\varUpsilon}}
\def\vX{{\varXi}}

\def\A{{\mathbb A}}

\def\J{{\mathbb J}}

\def\frm{{\mathfrak m}}

\def\Sh{{\sf Sh}}

\def\Spec{{\sf Spec}}

\def\Sets{{\mathscr{Sets}}}

\def\id{{\sf id}}

\def\Hom{{\sf Hom}}

\def\qc{{\mathfrak Q}{\mathfrak c}}
\def\Loc{\mathfrak{Loc}}

\def\hom{{\sf Hom}}

\numberwithin{equation}{subsection}

\begin{document}

\title{Reconstruction theorem for monoid schemes}

\author[I. Pirashvili]{Ilia Pirashvili}

\maketitle

\begin{abstract} We aim to reconstruct a monoid scheme $X$ from the category of quasi-coherent sheaves over it. This is much in the vein of Gabriel's original reconstruction theorem.

Under some finiteness condition on a monoid schemes $X$, we show that the localising subcategories of the topos $\mathfrak{Qc}(X)$ of quasi-coherent sheaves on $X$ is in a one-to-one correspondence with open subsets of $X$, while the elements of $X$ correspond to the topos points of $\mathfrak{Qc}(X)$. This allows us to reconstruct $X$ from $\mathfrak{Qc}(X)$.
\end{abstract}
\vspace{1em}

{\bf \emph{Keywords}:} \emph{Commutative monoid, monoid scheme, topos, localising subcategory, topos points}
\\
\\
\emph{MSC classes:	18B25, 20M14, 20M30}

\section{introduction}

According to the classical result of Gabriel \cite{ab}, a Noetherian scheme $S$ can be reconstructed from the category of quasi-coherent sheaves over $S$. The aim of this work is to modify an argument by Gabriel, to show that a similar statement is true for a monoid scheme of finite type.

To state our result, let us fix some terminology and notations. See any book in topos theory, or the main body of this paper for the terms used here. In what follows, $\Sets$ denotes the category of sets, which is the terminal topos. For a category $\A$, one denotes the centre of $\A$ by ${\mathcal Z}(\A)$, which, by definition, is the commutative monoid of all natural transformation of the identity functor $\id_\A$ to itself. We denote the collection of all isomorphism classes of points of a topos $\ea$ by ${\sf F}(\ea)$, while $\Loc(\ea)$ denotes the poset of all localising subcategories of $\ea$. For a localising subcategory $\D$ of $\ea$ and $p=(p_*,p^*):\Sets \to\ea$ a topos point of $\ea$, we write $p\fork \D$ if $p_*(S)\in \D$ for every set $S\in\Sets$.

Let $X$ be a monoid scheme. The poset of all open subsets of $X$ is denoted by ${\ope}(X)$. The structure sheaf of $X$ is denoted by $\mathcal{O}_X$ and the category of quasicoherent sheaves over $X$ by ${\mathfrak Q}{\mathfrak c}(X)$. A monoid scheme is called of \emph{finite type} if it can be covered by a finite number of affine subschemes $\Spec(M_i)$, such that each $M_i$ is a finitely generated monoid.  

The following is our version of the reconstruction theorem:

\begin{Th}[{\bf Main Theorem}]\label{maintheorem} Let $X$ be a Noetherian monoid scheme. The following results hold:
	\begin{itemize}
		\item[i)] The category ${\mathfrak Q}{\mathfrak c}(X)$ is a topos.
		\item[ii)] The map $\Phi:{\ope}(X)\to \Loc({\mathfrak Q}{\mathfrak c}(X))$, given by
		$$\Phi(U)={\mathfrak Q}{\mathfrak c}(U),$$
		is an isomorphism of posets
		\item[iii)] If $U$ is an open subset of $X$ and $\D=\Phi(U)$ is corresponding localising subcategory of ${\mathfrak Q}{\mathfrak c}(X)$, then
		$$\mathcal{O}_X(U)= {\mathcal Z}(\D).$$
		\item[iv)] For any point $x\in X$, the functor
		$$\qc(X)\to\Sets, \quad \mathcal{A}\mapsto \mathcal{A}_x,$$
		which assigns to a quasicoherent sheaf $\mathcal{A}$ the stalk $\mathcal{A}_x$ at $x$, is the inverse part of a point of the topos $\qc(X)$. Furthermore, the corresponding map $\Pi:X \to {\sf F}(\qc(X))$ on the underlying topological space of $X$ is a bijection if $X$ is quasi-projective.
		\item[v)]  If $x\in X$ and $U$ is an open subset of $X$, then $x\in U$ if and only if $\Pi(x)\fork \Phi(U).$
	\end{itemize} 
\end{Th}

Parts i) and iv) were proven in \cite{points}. The aim of this work is to prove the rest.

\section{Preliminaries on topoi and localising subcategories}

\subsection{Grothendieck topologies and Grothendieck topoi}

Let $\A$ be a small category. A set valued presheaf (or simply presheaf) on $\A$ is a functor $P:\A^{op} \to \Sets$.

For an object $A$ in $\A$, a \emph{sieve} $S$ on $A$ is a subfunctor of the representable functor $\Hom_\A(-,A)$. Recall that a \emph{Grothendieck topology} $J$ on $\A$ assigns to each object $A$ in $\A$ a collection $JA$ of sieves on $A$, called \emph{covering sieves}, such that certain axioms hold. We refer to \cite{mm} (or \cite{hand3},\cite{topos}) for these axioms. However, we will list these axioms in Section \ref{tm}, for the special case when $\A$ is a one object category, corresponding to a monoid $M$.

A \emph{site} $(\A, J)$ consists of a small category $\A$, equipped with a Grothendieck topology $J$. A presheaf $F:\A^{op}\to \Sets$ is said to be a \emph{ sheaf} for the topology $J$ if, whenever we have a covering sieve $S$ of an object $A$, any morphism $\al:S\to F$ of presheaves has a unique extension $\Hom_\A(-,A)\to F$. The category of sheaves on $\A$ with respect to the topology $J$ is denoted by ${\sf Sh}(\A,J)$.

Recall that a \emph{Grothendieck topos} is a category $\ea$, equivalent to the category of set valued sheaves over a Grothendieck site.

If $\ea'$ and $\ea$ are topoi, a \emph{geometric morphism} $f:\ea'\to \ea$ is a pair of functors $f_*:\ea'\to \ea$ and $f^*:\ea\to \ea'$, such that $f^*$ is the left adjoint of $f_*$ and $f^*$ respects finite limits.

A \emph{point} of a topos $\ea$ is simply a geometric  morphism $p=(p_*.p^*):\Sets\to \ea$ from  the topos of sets to $\ea$.

For a small category $\A$, one denotes by $\Topo(\A)$ the collection of all Grothendieck topologies on $\A$. It is well-known that the inclusion induces an order on $\Topo(\A)$, which makes it  a  complete lattice \cite  [Lemma 0.34] {topos}. If $\J$ is the minimal (discrete) topology on $\A$, then the corresponding category $\ea={\sf Sh}(\A,J)$ is just the category $\fct(\A^{\op},\Sets)$ of all presheaves over $\A$. If $J$ is the maximal (indiscrete) topology, then $\ea={\sf Sh}(\A,J)$ is the trivial category, which has only one object and one morphism. For a site $(\A,J)$, we write $\Topo_J(\A)$ to denote the subclass of Grothendieck topologies containing $J$. This class is in one-to-one correspondence to two other classes, which can be described entirely in terms of the topos $\ea={\sf Sh}(\A,J)$. To describe them, we need some additional facts from topos theory.

\subsection{Subobject classifier}

It is well-known (see for example \cite{mm}), that any Grothendieck topos has a \emph{subobject classifier}. This is an object $\oO$, together with a morphism from the terminal object $t:1\to  \oO$, such that for any object $A$ and any subobject $i:B\subseteq A$, there is a unique morphism $\chi_B:A\to \oO$, for which
$$\xymatrix{ B\ar[r]^i\ar[d]_{_B!} &A \ar[d]^{\chi_B}\\ 1\ar[r]_t & \oO }$$
is a pullback diagram in $\ea$. Here and elsewhere, $_B!$ denotes the unique morphism from $B$ to the terminal. The map $\chi_B:A\to \oO$ is called the \emph{characteristic map} of  a subobject $B$. For example, $\chi_A=t\circ _A!$.

\subsection{Lawvere-Tierney topologies}

Let $\ea$ be a topos. Recall that a \emph{Lawvere-Tierney topology} on $\ea$ is given by a morphism $j:\oO\to \oO$, such that $j\circ j=j$ and the following diagrams commute:
$$\xymatrix{ 1\ar[r]^t \ar[rd]_t & \oO\ar[d]^j & & & & \oO\times \oO\ar[d]^{j\times j} \ar[r]^{\wedge} &\oO\ar[d]^j \\
			 & \oO & & & & \oO\times \oO\ar[r]^{\wedge} & \oO. }$$
Here, $\wedge:\oO\times \oO\to \oO$ is the characteristic map of the subobject $1=1\times 1 \xto{(t,t)}\oO\times \oO.$ The collection of all Lawvere-Tierney topologies on $\ea$ is denoted by ${\mathfrak L}{\mathfrak T}(\ea)$.

Let $j$ be a Lawvere-Tierney topology on $\ea$. An subobject $B\hookrightarrow A$ is called \emph{$j$-dense} if the following diagram commutes
$$\xymatrix{ & A\ar[ld]_{\chi_B} \ar[rd]^{\chi_A} & \\
			 \oO\ar[rr]^j & & \oO.}$$
A object $F$ of the topos $\ea$ is called a \emph{$j$-sheaf}, if for any object $A$ and any dense subobject $B\subseteq A$, the induced map
$$\Hom_\ea(A,F)\to \Hom_\ea(B,F)$$
is a bijection. Denote by $\Sh_j(\ea)$ the full subcategory of $\ea$ consisting of $j$-sheaves. It is well-known that $\Sh_j(\ea)$ is also a topos and the inclusion $\Sh_j(\ea) \to \ea$ has a left adjoint functor, which commutes with finite limits. 

\subsection{Localising subcategories}

Recall that a \emph{localising subcategory} $\D$ of a topos $\ea$ is a full subcategory of $\ea$, such that the following conditions hold:
 
\begin{itemize}
\item [ i)]  If $x$ belongs to $\D$  and $y\in \ea$ is  isomorphic to $x$, then $y$ belongs to $\D$.
\item [ ii)] The inclusion  $\iota:\D\to \ea$ has a left adjoint $\rho:\ea\to \D$, called the \emph{localisation}.
\item [ iii)] The localisation $\rho$  respects  finite limits.
\end{itemize}

If this is the case, then $\D$ is also a topos and the composite functor $\rho\circ \iota$ is isomorphic to  the identity functor $\id_\D$, see \cite[Proposition I.1.3]{GZ}. Moreover, the pair $(i,\rho)$ defines a geometric morphism $\D\to \ea$, known as \emph{embedding} \cite{mm}. Without loss of generality, we can always choose $\rho$ in such a way that $\rho\iota=\id_{\D}$.

We denote by $\mathfrak{Loc}({\ea})$ the collection of all localising subcategories of $\ea$. The collection $({\ea})$ has a natural poset structure induced by inclusion of subcategories.

The following well-known result explains the relationship between Grothendieck topologies and localising subcategories.

\begin{Le}\label{2510718} Let $(\A,J)$ be a site and $\ea={\sf Sh}(\A,J)$ the corresponding topos.

\begin{itemize}
\item[i)] For $K\in \Topo_J(\A)$, the category ${\sf Sh }(\A,K)$ is a localising subcategory of $\ea$ and the localising functor is the composite functor
$$\ea={\sf Sh}(\A,J)\subseteq \fct(\A^{op},\Sets)\to {\sf Sh }(\A,K),$$
where the last functor is the sheafification functor. 
\item[ii)] If $\D$ is a localising subcategory of $\ea$ with localisation $\rho:{\ea}\to \D$, we denote by $K$ the collection of all sieves $R\subseteq \Hom_\A(-,A)$ for which the induced morphism
$$\rho(R)\to \rho(\Hom_\A(-,A))$$
is an isomorphism. Then $K\in \Topo_J(\A)$.
\item[iii)] The correspondences defined in i) and ii) induces an order reversing bijection between $\Loc({\ea})$ and $\Topo_J(\A)$.
\item[iv)] The result remains true if one uses the Lawvere-Tierney topologies on $\ea$. That is, if $j$ is a Lawvere-Tierney topology on $\ea$, then $\Sh_j(\ea)$ is a localising subcategory of $\ea$. The localisation in this case is given by sheafification. In this way, one obtains an order reversing bijection
$${\mathfrak L}{\mathfrak T}(\ea)\cong \Loc({\ea}).$$
\end{itemize}
\end{Le}

This fact is well-known, see for example \cite[Propositions 3.5.4 and  3.5.7]{hand3}  for the  bijection in iii) and \cite[Exercise 2. Ch3]{topos} for the  bijection in iv). 

\subsection{Centre of a category}

One way to obtain a commutative  monoid from a category is to consider its centre. Recall that the \emph{centre} ${\mathcal Z}(\A)$ of a category $\A$ is the monoid of all natural transformation $\id_\A\to \id_\A$ of the identity functor to itself. It is well-known that  the monoid ${\mathcal Z}(\A)$ is commutative.

Thus, an element $\theta$ of the centre ${\mathcal Z}(\A)$ of $\A$ is a collection of morphisms $\theta_A:A\to A$, where $A$ is running through all the objects of $\A$, such that for any morphism $f:a\to b$ of the category $\A$, the diagram
$$\xymatrix{ a\ar[d]_f\ar[r]^{\theta_a} & a\ar[d]^f\\
			 b\ar[r]^{\theta_b} & b}$$
commutes.

Let $\ea$ be a topos. If $\D\subseteq \D'$ are localising subcategories of $\ea$ and $\theta$ is an element of ${\mathcal Z}(\D')$, then the collection $(\theta_x)_{x\in\D}$ belongs to ${\mathcal Z}(\D)$. Hence, one obtains the monoid homomorphism ${\mathcal Z}(\D')\to {\mathcal Z}(\D)$, called \emph{restriction}. This allows us to define a
$${\mathcal Z}:\Loc(\ea)^{op}\to {\sf Com.Monoids}.$$

\subsection{Gluing of topoi} Let ${\ea}_i$, $i=0,1,2$ be Grothendieck topoi and $\rho_i^*:{\ea}_i\to {\ea}_0$, $i=1,2$ localisations.
A \emph{gluing} (see \cite{points}) of ${\ea}_1$ and ${\ea}_2$ along ${\ea}_0$ is the category $\ea$, whose objects are triples $(A_1,A_2,\al)$, where $A_i$ is an object of ${\ea}_i$, $i=1,2$ and $\al:\rho_1^*(A_1)\to \rho_2^*(A_2)$ is an isomorphism in ${\ea}_0$. A morphism $(A_1,A_2,\al)\to (B_1,B_2,\beta)$ in ${\ea}$ is a pair $(f_1,f_2)$, where $f_i:A_i\to B_i$, is a morphism in ${\ea}_i$, $i=1,2$, such that $\rho_2^*(f_2)\al=\beta \rho_1^*(f_1)$.
We have the following diagram
$$\xymatrix{{\ea}\ar[r]^{pr_1}\ar[d]_{pr_2}&{\ea}_1\ar[d]_{\rho^*_1}\\ { \ea}_2\ar[r]_{\rho^*_2}&{\ea}_0,}$$
where $pr_i:{\ea}\to {\ea}_i$ is given by $pr_i(A_1,A_2,\al)=A_i$, $i=1,2$. 

Recall that an object $A$ of a Grothendieck topos ${\ea}$ is called \emph{$s$-Noetherian} if it satisfies the ascending chain condition on its subobjects. A Grothendieck topos ${\ea}$ is called \emph{locally $s$-Noetherian} provided it posses a system of s-Noetherian generators, see \cite[Section 2.4]{points} for more on this.

According to \cite{points}, $\ea$ is a locally $s$-Noetherian topos and the functors $pr_1$, $pr_2$ are localisations if  ${\ea}_i$, $i=0,1,2$, are locally $s$-Noetherian topoi.

\begin{Pro}\label{lnt} Let ${\ea}_i$, $i=0,1,2$ be locally $s$-Noetherian topoi, $\rho_i^*:{\ea}_i\to {\ea}_0$, $i=1,2$ localisations and $\ea$ be the corresponding gluing topos. The following is a pull-back diagram
$$\xymatrix{{\mathcal Z}(\ea) \ar[r] \ar[d]& {\mathcal Z}({\ea_1})\ar[d]\\
{\mathcal Z}({\ea_2})\ar[r] &{\mathcal Z}({\ea_0}).}$$

\end{Pro}

\begin{proof} The functors $\rho_i^*:\ea_i\to \ea_0$ have, by assumption, right adjoint functors $\rho_{i*}:\ea_0\to\ea_i$, $i=1,2$. Denote by $\kappa_i:\id_{\ea_i}\to \rho_{i*}\rho^*$, $i=1,2$ the counite of the adjoints. We may assume, without loss of generality, that $\rho^*_i\rho_{i*}(A_0)=A_0$ for any object $A_0$ of $\ea_0$.
Hence, the categories $\ea_0$, $\ea_1$ and $\ea_2$ can be identified respectively  with the following full subcategories of $\ea$:
$$\{(\rho_{1*}(A_0),\rho_{2*}(A_0), \id_{A_0})| A_0\in \ea_0\},$$
$$\{(A_1,\rho_{2*}\rho^*_1(A_1), \id_{\rho_{1*}A_1})| A_1\in \ea_1\},$$
$$\{(\rho_{1*}\rho^*_2(A_2),A_2, \id_{\rho_{2*}A_2})| A_2\in \ea_2\}.$$
We have to show that for any natural transformations $\xi_i:\id_{\ea_i}\to \id_{\ea_i}$, $i=1,2$, such that
$$\xi_1(\rho_{1*}(A_0))=\xi_2(\rho_{2*}(A_0))$$
for every object $A_0$ of $\ea_0$, there exists a unique natural transformation $\theta:\id_{\ea}\to \id_{\ea}$, such that
$$\theta(A_1,\rho_{2*}\rho^*_1(A_1), \id_{\rho_{1*}A_1})=(\xi_1(A_1), \rho_{2*}\rho^*_1\xi_1(A_1))$$
and
$$\theta(\rho_{1*}\rho^*_2(A_2),A_2, \id_{\rho_{2*}A_2})=(\rho_{1*}\rho^*_2\xi_2(A_2),\alpha_2(A_2))$$
holds for every $A_i\in\ea_i$, $i=1,2$.

To this end, take any object $(A_1,A_2,\alpha)$ of $\ea$. Since $\theta(A_1,A_2,\alpha)$ must be a morphism in $\ea$, we have
$$\theta(A_1,A_2,\alpha)=(\theta_1(A_1,A_2,\alpha),\theta_2(A_1,A_2,\alpha)),$$
where $$\theta_1(A_1,A_2,\alpha):A_1\to A_1\quad {\rm and} \quad \theta_2(A_1,A_2,\alpha):A_2\to A_2$$ are morphisms in $\ea_1$ and $\ea_2$ respectively. We have the following morphism in $\ea$:
$$(\id_{A_1}, \rho^*_2(\alpha^{-1})\circ \kappa_2(A_2)):(A_1,A_2,\alpha)\to (A_1,\rho_{2*}\rho^*_1(A_1), \id_{\rho_{1*}A_1}).$$
Applying the naturality of $\theta$ for this morphism, we obtain
$$\theta_1(A_1,A_2,\alpha)=\xi_1(A_1).$$
Similarly, we obtain $\theta_2(A_1,A_2,\alpha)=\xi_2(A_2)$. This, not only proves the uniqueness of $\theta$, but also yields the method to construct it and the result follows. \end{proof}

\section{Monoids, $M$-sets and subobject classifier}

We investigate some properties of the subobject classifier of the topos of $M$-sets, where $M$ is a monoid. The main result of this section says that if $M$ is $s$-Noetherian and commutative, then the subobject classifier respects localisations. This allows us to construct, in Theorem \ref{osch}, the subobject classifier of the topos $\qc(X)$ of quasi-coherent sheaves over $X$, where $X$ is a locally $s$-Noetherian monoid scheme.

\subsection{Subobject classifier for the topos of $M$-sets}

In this section, we will recall the construction of the subobject classifier of the topos of right $M$-sets $\Sets_M$, where $M$ is a multiplicatively written monoid \cite{mm}.

Elements of the set $\oO$ (also denoted as $\oO^{M}$, to make explicit the role of the monoid $M$) are right ideals of $M$.  We recall that a subset $\m\subseteq M$ is called a \emph{right ideal}, if $\m M\subseteq \m$. 
In order to describe the action of $M$ on $\oO$, for a right ideal $\m$ and an element $a\in M$, we set
$$(\m:a)=\{x\in M| ax\in \m\}.$$
\begin{Le}\label{21247} Let the notations be as above. One has:
\begin{itemize}
	\item[i)] $(\m:a)$ is a right ideal.
	\item[ii)] $(\m:1)= \m$,
	\item[iii)] $((\m:a):b)=(\m:ab),$
	\item[iv)] $(M:a)=M$.
	\end{itemize}
\end{Le}
\begin{proof} i) Take any element $m\in M$ and $x\in (\m:a)$. Then $ax\in \m$. Since $\m$ is a right ideal $(ax)m\in \m$. So $xm\in (\m:a)$.

ii) and iv) are obvious. 

iii) For an element $x\in M$, one has $x\in ((\m:a):b)$ if and only if $bx\in (\m:a)$. The last condition is equivalent to $abx\in \m$. In other words, $x\in ((\m:a):b)$ if and only if $x\in (\m:ab)$.
\end{proof}

\begin{Co} The rule $(\m,a)\mapsto (\m:a)$ defines a right $M$-set structure on $\oO$.
\end{Co}

By iv) of Lemma \ref{21247}, the map $t:1\to \oO$, given by $t(1)=M\in \oO$, is a morphism of $M$-sets. For any right $M$-set $A$ and any $M$-subset $i:B\subseteq A$, define the map $\chi_B:A\to \oO$ by
$$\chi_B(a)=\{m\in M| am\in B\}.$$
One easily sees that $\chi_B$ is a morphism of $M$-sets and the diagram
$$\xymatrix{ B\ar[r]^i\ar[d]_{_B!} &A \ar[d]^{\chi_B}\\ 1\ar[r]_t & \oO }$$
is a pullback diagram. Hence, $\oO$ is the subobject classifier in $\Sets_M$  \cite{mm}. 

\subsection{Grothendieck topologies on a monoid}\label{tm}

According to Lemma \ref{2510718}, there is an order reversing bijection between the set $\Loc(\Sets_M)$ and all Grothendieck topologies defined on the one object category corresponding to $M$. Let us recall the last notion in this particular case.

For a monoid $M$ (not nesseseraly commutative), a \emph{Grothendieck topology}, or simply  a \emph{topology}, $\fa$ on $M$ is a collection of right ideals, such that
\begin{itemize} 
\item[(T1)]  $M\in \fa$,
\item[(T2)]If $\af\in \fa$ and $m\in M$, then $(\af:m)\in \fa$,
\item[(T3)] If $\fb\in \fa$ and $\af$ is a right ideal of $M$ such that  $(\af:m)\in \fa$ for any $m\in \fb$, then $\af\in \fa$.
\end{itemize}

Thus, $\fa$ is a subobject of the $M$-set $\oO_M$. Since $\oO_M$ is a subobject classifier, it defines a map $\chi_\fa:\oO_M\to \oO_M$, which is a Lawvere-Tierney topology on $\Sets_M$ and any Lawvere-Tierney topology on $\Sets_M$ is of this form. Explicitly, the map $\chi_\fa$ is given by
$$\chi_\fa(\af)=\{m\in M| (\af:m)\in \fa\}.$$

This result draws a clear analogy to Gabriel's topologies in ring theory \cite{bo}. We will see tat many of the results originally proven for rings by Gabriel \cite{ab} are still valid for monoids.
 
\begin{Le}\label{e1} Let $\fa$ be a topology on $M$. The following hold:
\begin{itemize}
\item[(i)] If $\af\subseteq \fb$ are right ideals and $\af\in \fa$, then $\fb\in \fa$.
\item[(ii)] If $\af, \fb\in \fa$, then $\af\cap \fb\in \fa$.
\item[(iii)] If $\af, \fb\in \fa$, then $\af\fb\in\fa$.
\end{itemize}

\end{Le}
\begin{proof} i) Take any element $m\in \af$. Then $m\in \fb$ and $(\fb:m)=R\in \fa$. Hence, by (T3), we have $\fb\in \fa$.

ii) If $m\in \af$, then $(\af\cap \fb:m)=(\fb:m)\in \fa$ by (T2). Hence, $\af\cap \fb\in \fa$, by (T3).

iii) For any $m\in \af$, we have $\fb\subseteq (\af\fb:m)$. Hence, by i), we have $(\af\fb:m)\in \fa$ and (T3) implies that $\af\fb\in \fa$.
\end{proof}

For a monoid $M$, we let $\Topo(M)$ denote the set of all topologies on $M$. As was already mentioned, $\Topo(M)$ is closed under arbitrary intersections  and is a complete lattice, where the ordering is given by inclusion.  The least element is just $\{M\}$, while the greatest element is the set of all right ideals.

\subsection{Monoid homomorphisms and subobject classifier}

Let $f:M\to N$ be a monoid homomorphism. One has the following set maps between the subobject classifiers of $M$ and $N$:
$$f^*:\oO^N\to \oO^M, \  \ {\rm and}\ \  f_*:\oO^M\to \oO^N$$
defined by
$$f^*(\n)=f^{-1}(\n), \  \ {\rm and}\ \ f_*(\m)=\bigcup_{x\in \m}f(x)N.$$
We point out that $\bigcup\limits_{x\in \m}f(x)N$ is the right ideal generated by $f(\m)$.

\begin{Le}\label{020218}
For any $\m\in\oO^M$, $\n\in\oO^N$ and $a\in M$, one has
\begin{itemize}
	\item[i)] $\m\subseteq f^*f_*(\m), \ \ f_*f^*(\n)\subseteq \n$,
	\item[ii)] $f_*f^*f_*(\m)=f_*(\m), \ \ f^*f_*f^*(\n)=f^*(\n)$,
	\item[iii)] $f_*(\m :a)\subseteq (f_*(\m):f(a))$,
	\item[iv)] $(f^*(\n):a)=f^*(\n:f(a))$.
\end{itemize}
\end{Le}

\begin{proof} i) is clear and ii) is a formal consequence of i).

To prove iii), observe that $f_*(\m :a)$ is a right ideal of $N$, generated by elements of the form $f(x)$, where $ax\in \m$. It follows that $f(ax)\in f_*(\m)$. Hence, $f(x)\in (f_*(\m):f(a))$, showing iii).

For iv), observe that
\begin{align*}
(f^*(\n):a)&=\{x\in M| ax\in f^*(\n)\}\\
&=\{x\in M|f(ax)\in \n\}\\
&=\{x\in M|f(x)\in (\n:f(a))\}\\
&= f^*(\n:f(a)).
\end{align*}
\end{proof}

\begin{Co}
The map $f^*:\oO^N\to \oO^M$ (unlike $f_*:\oO^M\to \oO^N$) is a morphism of $M$-sets. Here and in what follows, the monoid $M$ acts on $\oO^N$ via $f$.
\end{Co}

\subsection{Localisation}

Let $M$ be a commutative monoid. Recall that if $S\subseteq M$ is a submonoid, one can form a new monoid $S^{-1}M$, called the \emph{localisation of $M$ by $S$}. Elements of $S^{-1}M$ are fractions $\frac{m}{s}$, where $m\in M$ and $s\in S$. By definition, $\frac{m_1}{s_1}=\frac{m_2}{s_2}$ if and only if there is an element $s\in S$, such that $m_1ss_2=m_2ss_1.$  Actually, this construction can be extended to $M$-sets in the obvious way. It is also obvious that if $A$ is a $M$-set, then $S^{-1}A$ becomes an $S^{-1}M$-set. The canonical map
$$f:M\to S^{-1}M,  \ \  f(m)=\frac{m}{1}$$
is a monoid homomorphism, called the localisation map.

\begin{Le}\label{080217} Let $S$ be a submonoid of a commutative monoid $M$ and let $f:M\to S^{-1}M$ be the corresponding localisation.
\begin{itemize}
	\item[i)] If $\n$ is an ideal of $S^{-1}M$, then $f_*f^*(\n)=\n$.
	\item[ii)] $f_*(\m :a)=(f_*\m:f(a))$.
	\item[iii)] $f^*f_*(\m)=\bigcup _{s\in S} (\m:s)$.
	\item[iv)]  For any $t\in S$, one has
$$(f^*f_*(\frm):t)=f^*f_*(\frm).$$
\end{itemize}
\end{Le}

\begin{proof} i) By part i) of Lemma \ref{020218}, it suffice to show that $\n\subseteq f_*f^*(\n)$. Take an element $x=\frac{n}{s}\in \n$. Since $\n$ is an ideal, we see that $sx\in \n$. Hence, $n\in f^*(\n)$ and $sx=f(n)\in f_*f^*(\n)$. Since $ f_*f^*(\n)$ is an ideal of $S^{-1}M$, we see that $x=f(n)\cdot \frac{1}{s} \in f_*f^*(\n)$ and i) follows.

ii) By part iii) of Lemma \ref{020218}, it suffices to show that $(f_*(\m):f(a))\subseteq f_*(\m :a) $. To this end, take an element $z\in (f_*\m:f(a))$. Since $z\in S^{-1}M$, we can write $z=\frac{m}{s}$, where $m\in M$ and $s\in S$.
By assumption, $zf(a)\in f_*(\m)$. So, $\frac{ma}{s}\in f_*(\m)$. Thus, there exist an element $u\in \m$ and $t\in S$, such that
$$\frac{ma}{s}=\frac{u}{t}.$$
This implies that $matt'=ust'$, for an element $t'\in S$. We have $z=\frac{m'}{s'}$, where $m'=mtt'$ and $s'=stt'$. Since $am'=amtt'=ust'\in \m$, we have $z\in (f_*(\m):a)$. This proves the result.

iii) We have $x\in f^*f_*(\m)$ if and only if $\frac{x}{1}\in f_*(\m)$. This is equivalent to saying that $\frac{x}{1}=\frac{m}{s}$, for some $m\in \m$, $s\in S$. The last condition holds if and only if $xt\in \m$ for some $t\in S$, i.e. when $x\in (\m:t)$ and the result follows.

iv)  Since  $\n\subseteq (\n:a)$ for any ideal $\n$ in a commutative monoid $M$, we only need to show $(f^*f_*(\frm):t) \subseteq f^*f_*(\frm)$. Take $z\in (f^*f_*(\frm):t)$, then
$$zt\in f^*f_*(\frm)=\bigcup _{s\in S} (\frm:s).$$
Thus there exists $s\in S$ such that $zts\in \frm$. Hence,
$$z\in (\frm:ts)\subseteq \bigcup _{r\in S} (\frm:r)= f^*f_*(\frm)$$
and iv) follows.
\end{proof}

\begin{Co}\label{342003} The map $f_*:\oO^M\to \oO^{S^{-1}M}$ is compatible with the monoid actions and as such, induces the map of $S^{-1}M$-sets
$$f_{*S}:S^{-1}\oO^M\to \oO^{S^{-1}M},$$
which is surjective.
\end{Co}

\begin{proof} The first assertion follows from the part ii) of Lemma \ref{080217}, while surjectivity property from  part i) of Lemma \ref{080217}.
\end{proof}

Thus, by construction, we have a commutative diagram
$$\xymatrix{ \oO^M\ar[rr] \ar[dr]_{f_*} & & S^{-1}\oO^M\ar[dl]^{f_{*S}}\\
			 &\oO^{S^{-1}M}, & }$$
where the horizontal map is the localization: $\m\mapsto \frac{\m}{1}$.

\subsection{$s$-Noetherian monoids}

Following \cite{points}, a monoid $M$ is called \emph{right $s$-Noetherian} if $M$ satisfies the ascending chain condition on right ideals Equivalently, any family of right ideals of $M$ contains a maximal member and this holds if and only if any right ideal of $M$ is finitely generated.

\begin{Exm}\label{snoe}
\begin{itemize}
	\item[i)]  Any finite monoid is right $s$-Noetherian.
	\item[ii)] Any group is right $s$-Noetherian.
	\item[iii)] Let $M$ be a commutative and finitely generated monoid. Then $M$ is $s$-Noetherian, see \cite{points}.
	\item[iv)] A commutative monoid $M$ is $s$-Noetherian if and only if $M_{\sf red}= M/M^\times$ is $s$-Noetherian \cite{points}. Here, $M^\times$ is the group of invertible elements of $M$.
\end{itemize}
\end{Exm}

\subsection{Subobject classifier of $s$-Noetherian monoids}

We will show that under the $s$-Noetherianness assumption, the functor
$$f_!:\Sets_M\to \Sets_{S^{-1}M}; \ f_!(A)=S^{-1}A$$
respects the subobject classifier. This is due to the following Proposition:

\begin{Pro}\label{oOloc} Let $M$ be a commutative and $s$-Noetherian monoid. The canonical map
$$f_{*S}:S^{-1}\oO^M\to \oO^{S^{-1}M}$$
is an isomorphism of $S^{-1}M$-sets for any submonoid $S\subseteq M$.
\end{Pro} 

\begin{proof} The above map is surjective by Corollary \ref{342003}. Before we prove injectivity, we need to prove two auxiliary claims.

Our first claim is that for any ideal $\frm$, there is an element $s_\frm$, such that
$$f^*f_*(\frm)=(\frm:s_\frm).$$
We have  $f^*f_*(\frm)=\bigcup _{s\in S} (\frm:s)$ by part iii) Lemma \ref{080217}. Since $M$ is $s$-Noetherian, the ideal $\bigcup _{s\in S} (\frm:s)$ is finitely generated, say by $x_1\in (\frm:s_{1}),\cdots, x_k\in (\frm:s_{k})$. Take $s_\frm=s_1\cdots s_k$. Then $x_i\in (\frm:s_\frm)$ for all $i=1,\cdots, k$. Hence, $(\frm:s_\frm)$ contains the ideal generated by $x_1, \cdots,x_k$, which is $f^*f_*(\frm)$.

For any ideal $\m$ of $M,$ we have $\m\in\oO^M$, which is an $M$-set. So, we can also consider $\frac{\m}{1}$ as an element of $S^{-1}\oO^M$.

Our second claim is that the equality
$$\frac{\frm}{1}=\frac{f^*f_*(\frm)}{1}$$
holds in $S^{-1}\oO^{M}$. To show this fact, observe that by Claim 1 and part iv) of  Lemma \ref{080217}, we have $$(\frm:s_\frm)=f^*f_*(\frm)=(f^*f_*(\frm):s_{\frm})$$
and the second claim follows.

This second claim shows that the composite of the map $f^*:\oO^{S^{-1}M}\to \oO^{M}$ with the canonical morphism $\oO^M\to S^{-1}\oO^M$ is the inverse of $f_{*S}:S^{-1}\oO^M\to \oO^{S^{-1}M}$.
\end{proof}

\subsubsection{The nessecity of s-Noetherianness}

More generally, if $f=(f_*,f^*):\ta\to\ka$ is a geometric morphism from a topos $\ta$ to a topos $\ka$, there exist a canonical morphism $\tau:f^*(\Omega^\ka)\to \Omega^\ta$, classifying the monomorphism $f^*(t) :f^*(1)\to f^*(\Omega^\ta)$. According to \cite{open_j}, $f$ is called \emph{sub-open} if $\tau$ is a monomorphism. Hence, if $M$ is an s-Noetherian monoid, the localisation $S^{-1}:\Sets_M\to \Sets_{S^{-1}M}$ is the inverse image part of a sub-open geometric morphism $\Sets_M\to \Sets_{S^{-1}M}$. It should be point-out that the s-Noetherian assertion is essential.

In fact, take $M$ to be the multiplicative monoid of the natural numbers. For $S=M$, the map $\tau$ is not a monomorphism. To prove this claim, take $A=\Omega^M$ and $B=\Omega^{S^{-1}M}$, where $\alpha(I)=S^{-1}I$, $I\in\Omega^M$ as in the next Lemma. Moreover, let $a_1,a_2\in \Omega^M$ be the following elements: $a_1=M$, while $a_2=J$ consists of those natural numbers $=\prod_pp^{v_p}$, for which either $v_p=0$ or $v_p\geq p$. Here, $p$ runs through the set of all prime numbers. In this case, $\alpha(a_1)=\alpha(a_2)$. We claim that there is no natural number $s$, for which $M=J:s$. The Lemma below will imply the non-injectivity of $\tau$. To prove the assertion, assume that such an $s$ exists. Choose a prime $p$, such that $p$ does not divide $s$. Then $p\in M$, but $p\not =sn$, where $n\in J$. 

\begin{Le} Let $M$ be a commutative monoid and $S\subset M$ a submonoid. Moreover, let $A$ be an $M$-set and $B$ an $S^{-1}M$-set, which is considered as an $M$-set via the canonical homomorphism $M\to S^{-1}M$. Denote by $\beta:A\to S^{-1}A$ the canonical map. For a morphism of $M$-sets $\alpha:A\to B$, the following two conditions are equivalent

i) The morphism of $S^{-1}M$-sets $\tau: S^{-1}A \to B$ is a monomorphism. Here, $\tau(\frac{a}{s})= \frac{\alpha(a)}{s}$. 

ii) If $\alpha(a_1)=\alpha(a_2)$, there exist $s\in S$, such that $a_1s=a_2s$.

\end{Le}

\begin{proof} i) $\Longrightarrow$ ii). If $\alpha(a_1)=\alpha(a_2)$, then $\beta(a_1)=\beta(a_2).$ This follows from the injectivity of $\tau$, since $\alpha=\tau\beta$. So, $\frac{a_1}{1}=\frac{a_2}{1}$ and $a_1s=a_2s$ for some $s\in S$. 

ii) $\Longrightarrow$ i). Assume $\tau(\frac{a_1}{s_1})=\tau(\frac{a_2}{s_2})$. Thus, $\frac{\alpha(a_1)}{s_1}=\frac{\alpha(a_2)}{s_2}$. So, $\alpha(a_1s_2s)=\alpha(a_2s_1s)$ for some $s\in S$. By ii), this implies $a_1s_2st=a_2s_1st$ for some $t\in S$ and $\frac{a_1}{s_1}=\frac{a_2}{s_2}$. The injectivity of $\tau$ follows.
\end{proof}

\section{The case of affine monoid schemes}

In this section, we aim to prove Theorem \ref{maintheorem} (the main theorem) for the affine case, i.e, when $X=\Spec(M)$. In this case, the topos $\qc(X)$ is the category  $\Sets_M$ of $M$-sets. As we already mentioned we need to prove only the parts ii), iii) and v) of Theorem \ref{maintheorem}.

We start with a few general remarks on ordered sets, the order topology and the relationship between the order- and the Zariski topology.

\subsection{Ordered sets}

Let $P$ be a poset. For an element $a\in P$, we set
$$P_a=\{x\in P| a\leq x\}$$
and
$$P^a=\{x\in P| x\leq a\}.$$
Any poset $P$ has a natural topology, called the \emph{order topology}, where a subset $S\subseteq P$ is open if $y\in P$ and $x\leq y$ imply $x\in P$. Thus, $\ope(P)$ is a distributive lattice and it is finite if $P$ is finite. It is well-known, that any finite distributive lattice $L$ is of this form for a uniquely defined $P$ (see for example \cite[p.106]{enum}), namely for $P={\sf Irr}(L)$, the subset of irreducible elements of $L$ (an element $x\in L$ is irreducible if $x=y\vee z$ implies $x=y$ or $x=z$).

Recall that a \emph{Galois connection} between posets $X$ and $Y$ is a pair $\gamma=(\al, \bb)$, where $\al:X\to Y$ and $\beta:Y\to X$ are maps such that 
$$(x_1\leq x_2)\Longrightarrow \left (\al(x_1)\geq \al(x_2)\right), \ \ (y_1\leq y_2)\Longrightarrow\left(\bb(y_1)\geq \bb(y_2)\right)$$
and for any $x\in X$, $y\in Y$, one has
$$\left(x\leq \beta(y)\right )\Longleftrightarrow\left(\al(x)\leq y\right).$$
It follows that
$$x\leq \bb\al(x) \quad{\rm and} \quad y\leq \al\bb(y).$$
An element $x\in X$ (resp. $y\in Y$) is called \emph{stable} under the Galois connection,  if $x=\bb\al(x)$ (resp. $y=\al\bb(y)$). In this case, $\al$ and $\bb$ induce bijections on stable elements. 

\subsection{The comparison between the order- and the Zariski topology}

Recall that a proper ideal $\pf\subsetneq M$ of a commutative monoid $M$ is called \emph{prime} if for any $a,b\in M$, such that $ab\in \pf$, one has $a\in \pf$ or $b\in \pf$ \cite{deitmar}, \cite{p1}. This implies that, $M\setminus \pf$ is  a submonoid of $M$. We denote the set of all prime ideals of $M$ with $\Spec(M)$. Let $\af$ be an ideal of $M$. We set
$$V(\af)=\{\pf\in {\Spec}(M)| \af\subseteq \pf \}.$$

It is classical to equip the set ${\Spec}(M)$ with a topology, called the \emph{Zariski topology}, where the $V(\af)$'s are closed subsets. Here, $\af$ runs through all the ideals of $M$.

For any element $f\in M$, one puts $$D(f)=\{\pf\in \Spec(M) | f\notin \pf\}.$$ It is well-known, that the sets $D(f)$ form a  basis of open sets for the  Zariski topology on $\Spec(M)$. 

Since $\Spec(M)$ is also a poset under inclusion, it also has the order topology. Any Zariski open subset is open in the order topology. To see this, we can restrict ourself with subsets of the form $D(f)$.  If $\pf\in D(f)$ and $\qf\subseteq \pf$, then $f\not \in \pf$ and hence, $f\not \in \qf$. Thus, $\qf\in D(f)$. It follows that, $D(f)$ is open in the order topology. The converse is not true in general, but it is true in the following important case.

In what follows the group of invertible elements of $M$ is denoted by $M^\times$. 

\begin{Le}\label{z=or} If $M/M^\times$ is finitely generated, the Zariski topology on   $\Spec(M)$ coincides with the order topology.
\end{Le}

\begin{proof} We have to show that for any  prime ideal $\pf$, the set $\{\qf\in \Spec(M)| \qf\subseteq \pf\}$ is open in the Zariski topology. Let $x_1,\cdots,x_k\in M$ be the representatives of generators of $M/M^\times$. Without loss of generality, we can assume that $x_1,\cdots,x_m\not\in \pf$ and $x_{m+1},\cdots,x_k \in \pf$. Denote by $S$ the submonoid of $M$ generated by invertible elements and by $x_1,\cdots, x_m$. Since $\pf$ is prime, $M\setminus \pf$ is a submonoid containing the invertible elements of $M$ and, also, the elements $x_1,\cdots, x_m$. Hence, $S\subseteq M\setminus \pf.$ Take any $y\not \in\pf$. It can be decomposed as a product of an invertible element and elements of $x_i$. Clearly, $1\leq i \leq m$ as otherwise $y$ would be an element in $\pf$, since $\pf$ is an ideal. This shows that $S=M\setminus \pf$. Take $f=x_1\cdots x_m$. We claim
$$D(f)=\{\qf| \qf\subseteq \pf\},$$
which obviously implies the result. To show the claim, observe that 
 $f\not \in \pf$ (because $\pf$ is prime) and hence, $\pf\in D(f)$. It follows that 
$$\{\qf| \qf\subseteq \pf\}\subseteq D(f).$$
Conversely, take $\qf\in D(f)$. Then $f\not \in \qf$. Assume $\qf \cap S\not = \emptyset$. The product $\prod_{i=1}^m x_i$ is in $\qf$. Since $\qf$ is prime, it follows that $x_i\in \pf$ for some $1\leq i\leq m$. Thus, $f\in \qf$, which contradicts of our choice of $\qf$. Hence, $\qf\subseteq M\setminus S=\pf$ and the claim is proven. 
\end{proof} 

\subsection{Grothendieck topologies on a commutative monoid $M$ and subsets of $\Spec(M)$}

The goal of this subsection is to construct a Galois connection between the posets of Grothendieck topologies on a commutative monoid $M$ and subsets of $\Spec(M)$.

We start with the following observation.
    
\begin{Pro}\label{mt} Let $M$ be a commutative monoid and $\pf$ be prime ideal of $M$. We set $S_\pf=M\setminus \pf$. The following assertions holds:
\begin{enumerate} 
\item [(i)] The set
$$\fa_{\{\pf\}}:=\{\af\in \oO_M \ \ {\rm and } \ \ \af\cap S_\pf\not = \emptyset\}$$
is a topology on $M$.
\item[(ii)] For an ideal $\af$, one has 
$$\left(\af\in \fa_{\{\pf\}}\right) \Longleftrightarrow  \left( \pf\not\in V(\af)\right).$$
\item[(iii)] For prime ideals $\pf$ and $\qf$, one has
$\pf \subseteq \qf$ if and only if $\fa_{\{\qf\}} \subseteq \fa_{\{\pf\}}$.
\end{enumerate}
\end{Pro}

\begin{proof} i) Since $1\in S= M \cap S$, we obtain $M\in \fa_{\{\pf\}}$. Assume $\af\in \fa_{\{\pf\}}$ and $m\in M$. By assumption, there exists an element $s\in S$, such that $s\in \af$. Since $s\in (\af:m)$, the axiom (T2) follows. In order to verify (T3), let  $\fb$ be an ideal from $ \fa_{\{\pf\}}$ and let $\af$ be an ideal, such that $(\af:m)\in \fa_{\{\pf\}}$ for any $m\in \fb$. By assumption, there are $s,t\in S$, such that $s\in \fb$ and $t\in (\af:s)$. Thus, $st\in \af$. Since $st\in S$, we see that  $\af\cap S\not=\emptyset$ and axiom (T3) holds.

ii) Obvious.

iii) Assume $\pf \subseteq \qf$ and $\af\in \fa_{\{\qf\}}$. Then
$$\emptyset \not = \af\setminus \qf\subseteq \af\setminus\pf.$$
Hence, $\af\in \fa_{\{\pf\}}$. It follows that $\fa_{\{\qf\}} \subseteq \fa_{\{\pf\}}$.

Conversely, assume $\fa_{\{\qf\}} \subseteq \fa_{\{\pf\}}$. Take an element $m\in \pf$. For the principal ideal $mM$ we have $mM\not \in \fa_{\{\pf\}}$. It follows that $mM\not \in \fa_{\{\qf\}}$. Hence $mM\subseteq \qf$. In particular $m\in \qf$. Thus $\pf \subseteq \qf$.
\end{proof}
 
Recall our discussion on topologies on a monoid in Section \ref{tm}. Denote by ${\mathfrak P}(X)$ the set of  all subsets of a set $X$. We construct two maps
$$\vU: {\mathfrak P}(\Spec(M)) \to \Topo(M) \quad{\rm and}\quad \vX:  \Topo(M)\to{\mathfrak P}(\Spec(M)),$$
given by
$$\vU({\mathcal P}): =\bigcap_{\pf\in {\mathcal P}}\fa_{\{\pf\}} \quad{\rm and}\quad \vX(\fa)=\{\pf\in\Spec(M)| \pf\not \in \fa\},$$
where $\mathcal P\subseteq \Spec(M)$ is a subset of the set of prime ideals of $M$ and $\fa$ is a topology on $M$.

\begin{Pro}\label{06012020} \begin{itemize}

\item [i)] For an ideal $\af$, one has
$$\left(\af\in \vU({\mathcal P}) \right) \Longleftrightarrow \left(  \pf\not\in V(\af) \quad {\rm for} \quad {\rm all} \quad \pf\in {\mathcal P}\right)$$
and hence,
$$\vU({\mathcal P})=\{\af |V(\af)\cap {\mathcal P}=\emptyset\}.$$
\item [ii)] The functions $\vU$ and $\vX$ form a Galois correspondence between the posets ${\mathfrak P}(\Spec(M))$ and $\Topo(M)$.
\item[iii)] Let $\fa$ be a topology on $M$. Then $$\vU(\vX(\fa))
=\{\af \in \oO_M|V(\af)\subseteq \fa \}.$$
\item[iv)] Let ${\mathcal P}$ be a subset of $\Spec(M)$. Then
$$\vX(\vU({\mathcal P}))=\{\qf\in\Spec(M)| \exists \pf\in {\mathcal P}, \qf\subseteq \pf\}.$$
\end{itemize}
\end{Pro}

\begin{proof} i) follows from part ii) of Proposition \ref{mt}.

ii) Clearly both maps reverse the ordering. Thus, we have 
to show that if $\fa$ is a topology on $M$ and $\mathcal P$ is a subset of ${\Spec}(M)$, then
$$(\fa\subseteq \vU({\mathcal P}) )\Longleftrightarrow  ({\mathcal P}\subseteq \vX(\fa)).$$
In fact, $\fa\subseteq \vU({\mathcal P})$ holds if and only if for any ideal $\af\in \fa$ and any prime ideal $\pf\in \mathcal P$, one has $\pf \not \in V(\af)$. Since $\pf\in V(\pf)$, it follows that $\pf\not \in \fa$ and ${\mathcal P}\subseteq \vX(\fa)$.

On the other hand, assume ${\mathcal P}\subseteq \vX(\fa)$. For any $\pf\in \mathcal P$ one has $\pf\not \in \fa$. Take an ideal $\af\in\fa$. If $V(\af)\cap {\mathcal P}\not =\emptyset$, it follows that there exists $\pf\in {\mathcal P}$, such that $\af\subseteq \pf$. Hence, $\pf\in \fa$. This contradicts our assumption on ${\mathcal P}$. We have 
$V(\af)\cap {\mathcal P} =\emptyset$ and $\af\in \vU({\mathcal P})$, showing $\fa\subseteq \vU({\mathcal P})$ and ii) follows.

iii) For an ideal $\af$, one has
$$\left( \af\in \vU(\vX(\fa)) \right ) \Longleftrightarrow  \left( \forall \pf\in \vX(\fa)  \Longrightarrow\ \ \pf\not \in V(\af)\right )\Longleftrightarrow \left ( \pf\not\in \fa \Longrightarrow \pf\not \in V(\af)\right ).$$
The last condition is the same as $V(\af)\subseteq \fa$ and we are done.

iv) We have $\qf\in \vX(\vU({\mathcal P}))$ if and only if $\qf\not \in \vU({\mathcal P})$. By i), this happens if and only if there exists $\pf\in {\mathcal P}$, such that $\qf\subseteq \pf$. 
\end{proof}

\begin{Exm}
\begin{itemize}
	\item[i)] We have
	$$\vU({\{\emptyset\}})=\{{\rm all \ nonempty \ ideals \ of } \ M \}=\oO^M\setminus\{\emptyset\},$$
	while
	$$\vU({\emptyset})=\oO^M=\{{\rm all \  \ ideals \ of} \ M\}.$$
\begin{itemize}
	\item[ii)] If $\frm(M)$ is the ideal of all non invertible elements of $M$, which is the greatest proper ideal of $M$, then
	$$\vU({\{\frm(M)\}})=\{M\}$$
	and
	$$\vU({{\Spec}(M)})=\{M\}.$$
\end{itemize}
\end{itemize}
\end{Exm}
 
\begin{Le}\label{a290120} Let $f\in M$ be an element. We have
$$\vU(D(f))=\{\af\in \oO^M| f\in \af\}.$$
Moreover, the localised category corresponding to the topology  $\vU(D(f))$ via iii) of Lemma \ref{2510718} is the category of $M_f$-sets.
\end{Le}

\begin{proof} For the first assertion, observe that, by part i) of Proposition \ref{06012020}, we have $\af\in\vU(D(f))$ if and only if $V(\af)\cap D(f)=\emptyset$. In other words, if for any prime ideal $\pf$ containing $\af$, it also contains $f$. Obviously, this condition holds if $f\in \f$. The converse is also true because if $f\not \in \af$, the ideal, which is maximal among the ideals containing $\af$ and not containing any power of $f$, is prime.

For the second part, we have to show that an $M$-set $A$ is a sheaf with respect to this topology if and only if the map $l_f:A\to A$, given by $l_f(a)=fa$, is a bijection. 

Assume $A$ is a sheaf in this topology. This means that for any ideal $\af$ such that $f\in \af$, the restriction map
$$A=\hom_{M}(M, A)\to \hom_M(\af,A)$$
is a bijection. By considering $\af=fM$, injectivity of this map implies the injectivity of $l_m:A\to A, \ a\mapsto ma$. For the surjectivity of $l_f$, take any element $a\in A$. Let $m,n\in M$ be elements such that $fm=fn$. We have $fma=fna$ and hence, $ma=na$, since $l_f$ is injective. This shows the existence of a well-defined morphism of $M$-sets $fM\to A$ for which $fm\mapsto ma$. By the bijectivity of the map $A=\hom_{M}(M, A)\to \hom_M(\af,A)$, it follows that there exist $b\in A$, such that $fb=a$. Thus, $l_f$ is a bijection. 

Conversely, assume that $l_f:A\to A$ is bijective. This condition obviously implies the bijectivity of the restriction map $A=\hom_{M}(M, A)\to \hom_M(\af,A)$. Let us take any ideal $\af$, such that $f\in\af$. To show the bijectivity of $A=\hom_{M}(M, A)\to \hom_M(\af,A)$, observe that it fits in the diagram
$$A=\hom_{M}(M, A)\to \hom_M(\af,A)\to \hom_M(fM,A).$$
Since the composite map is a bijection, we only need to show injectivity of the second map. For this, take any morphism of $M$-sets $\alpha:\af\to A$. We have $\alpha(fx)=f\alpha(x)$ for any $x\in \af$. Thus
$$\alpha(x)=l_f^{-1}(\alpha(fx)),$$
showing that $\alpha$ is uniquely defined by its restriction on $fM$ and the result follows.
\end{proof}

Recall that a topology $\fa$ (resp. a subset $\mathcal P$) is stable if and only if $\fa=\vU\vX(\fa)$ (resp. ${\mathcal P}=\vX\vU({\mathcal P})$). From the properties of  Galois correspondences, one obtains a bijective correspondence between stable topologies on $M$ and stable subsets of $\Spec(M)$. Our next goal is to find a convenient way of characterising stable subsets and stable topologies.

\begin{Pro}\label{134'} Let $M$ be a commutative monoid. A subset ${\mathcal P}\subseteq \Spec(M)$ is stable if and only if it is an open subset in the order topology of $\Spec(M)$.
\end{Pro}
 
\begin{proof} This follows from part iv) of Proposition \ref{06012020}.
\end{proof}

\begin{Pro}\label{134} Let $\fa$ be a topology on a commutative monoid $M$. Then $\fa$ is stable if and only if for any ideal $\af\not\in \fa$, there exists $\pf\in V(\af)$, such that $\pf\not \in \fa$.
\end{Pro}

\begin{proof} The `if' part is a consequence of Proposition \ref{06012020}. For the other side, it suffices to show that $\vU\vX(\fa)\subseteq \fa$. Take an ideal $\af\not \in \vU\vX(\fa)$. By Proposition \ref{06012020}. we have $V(\af)\not \subseteq \fa$. So, there is a prime $\pf\in V(\af)$, such that $\pf\not \in \fa$.
\end{proof}

\begin{Co}\label{134''} For any commutative monoid $M$, there exist a bijection between the stable topologies of $M$ and open subsets of ${\Spec}(M)$ in the order topology.
\end{Co}

\begin{proof} This is a formal consequence of Galois correspondence and Proposition \ref{134'}.
\end{proof}

\begin{Le}\label{135} Let $\fa$ be a topology on a commutative monoid $M$.
\begin{itemize}
\item[(i)] Let $\cf$ be an ideal of $M$. If $\af$ is an ideal which is maximal with respect to the property that $\af\not\in \fa$ and $\cf\subseteq \af$, then $\af$ is a prime ideal.
\item[(ii)] Assume that for any ideal $\fb\in \fa$, there exists a finitely generated ideal $\af$, such that $\af\subseteq \fb$ and $\af\in \fa$. For any ideal $\cf\not \in \fa$, there exist a prime ideal $\pf\in V(\cf)$, such that $\pf\not \in \fa$.
\end{itemize}
\end{Le}

\begin{proof} (i) Take elements $a,b\not \in \af$ and assume $ab\in \af$. Consider the ideals $\af'=\af\cup aM$ and $\fb'=\af\cup bM\in \fa$. By our assumption, $\af',\fb'\in \fa$. We have $\af'\fb'\subseteq \af$, which contradicts the fact that $\af'\fb'\in \fa$. The latter holds due to Lemma (\ref{e1}), implying that $\af$ is prime.

(ii) Consider the set $\mathcal{A}$ of all ideals $\af$, such that $\cf\subseteq \af$ and $\af\not \in \fa$. Clearly, $\mathcal{A}$ admits an ordering by inclusion and $\cf\in \mathcal{A}$. We will use Zorn's lemma.

Take any chain $(\af_i)_{i\in I}$ in $\mathcal{A}$. Then $\fb=\bigcup_i\af_i$ is an ideal since any union of ideals of a monoid is an ideal. We claim that $\fb\not \in \fa$. Assume $\fb\in \fa$. We can find a finitely generated ideal $\fb'\in \fa$, such that $\fb'\subseteq \fb=\bigcup_i\af_i$. Since $\fb'$ is finite generated,  $\fb'\in \af_i$ for some $i$. This contradicts the fact that $\af_i\in \mathcal{A}$, proving the claim. Now we can use Zorn's lemma to conclude that $\mathcal{A}$ has a maximal element $\pf$, which is prime thanks to (i) and $\cf\subseteq \pf$.
\end{proof}

\begin{Le}\label{top} Let $M$ be an $s$-Noetherian monoid. Any topology $\fa$ on $M$ is stable and hence, there exists a subset ${\mathcal P}\subseteq {\Spec}(M)$, such that $\fa=\fa_{\mathcal P}$.
\end{Le}

\begin{proof} By  Proposition (\ref{134}), we have to show that for any ideal $\cf\not \in \fa$, there exist a prime $\pf\in V(\cf)$, such that $\pf\not \in \fa$. This follows from the second part of Lemma (\ref{135}).
\end{proof} 

This fact, together with Corollary \ref{134''}, implies the following:

\begin{Co}\label{bij_or} Let $M$ be an $s$-Noetherian monoid. There exists a bijection between the topologies of $M$ and open subsets of $\Spec(M)$ in the order topology.
\end{Co}

\begin{Co}\label{26} Let $M/M^\times$ be a finitely generated monoid and $\ea$ the category of $M$-sets. There is an isomorphism
$$ \Loc({\ea})\cong \ope(\Spec(M)).$$
\end{Co}

\begin{proof} We know (see Lemma \ref{2510718}, part iii)) that $ \Loc({\ea})\cong \Topo(M)$ (this is valid for all $M$).  Hence, the result follows from Lemma \ref{z=or} and Corollary \ref{bij_or}.
\end{proof}

\section{The general case}

\subsection{ Preliminaries on monoid Schemes}

Recall that for any commutative monoid $M$ and any $M$-set $A$,  there exist a unique sheaf $\tilde{A}$ of sets on  $\Spec(M)$, such that
$$\Gamma(D(f),\tilde{A})=A_f.$$
Here and elsewhere, $\Gamma(U, \mathcal{F})$ denotes the set of section of a sheaf $\mathcal{F}$ on an open subset $U$ (i.e. $\mathcal{F}(U)$) and $A_f$ is the localisation of $A$ with respect to a submonoid of $M$ generated by $f$. The stalk of $\tilde{A}$  at the point $\pf$ is the localisation $A_\pf$ of $A$ by the submonoid $M\setminus \pf$. If $A=M$, the sheaf $\tilde{M}$ is denoted by ${\mathcal O}_M$. The sheaf ${\mathcal O}_M$ is the sheaf of monoids, while the sheaf $\tilde{A}$ becomes a sheaf of  ${\mathcal O}_M$-sets. The pair $(\Spec(M), {\mathcal O}_M)$ is called an \emph{affine monoid scheme}.

Like in classical algebraic geometry, one can glue affine monoid schemes to obtain \emph{monoid schemes}.

A monoid scheme is called $s$-\emph{Noetherian}, if it can be covered by a finite number of open affine monoid subschemes $Spec(M_i)$, where each $M_i$ is an $s$-Noetherian monoid.  A monoid scheme is called \emph{Noetherian} or \emph{of finite type}, if it can be covered by a finite number of open affine monoid subschemes $Spec(M_i)$, where each $M_i$ is a finitely generated monoid.

Let $X$ be a monoid scheme and $\mathcal{A}$ a sheaf of $\mathcal{O}_X$-sets. That is, a sheaf of sets, together with an action of the monoid $\mathcal{O}_X(U)$ on the set $\mathcal{A}(U)$ for all open  $U\subseteq X$, such that these actions are compatible when $U$ varies. One denotes by $\mathcal{O}_X$-${\sf Sets}$ the category of sheaves of $\mathcal{O}_X$-sets.  

Let $\mathcal{A}$ be a sheaf of $\mathcal{O}_X$-set. It is \emph{quasicoherent} \cite{Lorscheid}, \cite{points} if for any point $x\in X$, there is an open and affine subscheme $U=Spec(M)$, such that $x\in U$ and the restriction of $\mathcal{A}$ on $U$ is isomorphic to a sheaf of the form $\tilde{A}$, for an $M$-set $A$. The category of all quasicoherent sheaves on $X$ is denoted by $\qc(X)$. 

According to \cite{points}, the category of quasi-coherent sheaves over an $s$-Noetherian scheme is a locally $s$-Noetherian topos \cite{points}. The next result describes the subobject classifier of this topos.

\begin{Th}\label{osch} Let $X$ be an $s$-Noetherian monoid scheme. There exists a quasi-coherent sheaf $\oO^X$, such that the restriction of $\oO^X$ on an affine and open monoid subscheme $V=\Spec(M)$ is $\oO^M$. The sheaf $\oO^X$ is the subobject classifier of the topos ${\mathfrak Q}{\mathfrak c}(X)$.
\end{Th}

\begin{proof} The existence of such a sheaf follows from Proposition \ref{oOloc}. Take quasi-coherent sheaves $\mathcal{B}$ and $\mathcal{A}$ and assume that $\mathcal{A} \subseteq \mathcal{B}$. We know that $\oO^M$ is a subobject classifier in the topos of $M$-sets. Hence, locally, there exist a unique morphism of sheaves $:\xi_{\mathcal{B}}:\mathcal{A}\to \oO^X$, as in the definition of the subobject classifier. Uniqueness implies that these local maps agree on the intersection of affine open subsets. This shows that they are restrictions of a uniquely defined morphism of sheaves $\mathcal{A}\to \oO^X$.
\end{proof}

\subsection{The centre of $\qc(X)$}

\begin{Le}\label{msz} Let $X$ be an $s$-Noetherian monoid scheme. Then $${\mathcal Z}(\qc(X))\cong \Gamma(X,\mathcal{O}_X).$$
\end{Le}

\begin{proof} We first consider the case when $X=\Spec(M)$ is affine. In this case, the result is classical and it follows basically from the Yoneda lemma, since for any $M$-set $A$, we have a functorial isomorphism
$$A\cong \Hom_M(M,A).$$
This shows that $M$ represents the identity functor on the category of $M$-sets. For a general monoid scheme $X$, observe that if $U$ and $V$ are open monoid subschemes of $X$, the category $\qc(U\cup V)$ is equivalent to the gluing of $\qc(U)$ and $\qc(V)$ along $\qc(U\cap V)$. Lemma \ref{lnt} finishes the proof.
\end{proof}

\subsection{The proof of main theorem}\label{proof_mth}

The aim of this subsection is to prove the main Theorem, described in the introduction. We will need one preliminary lemma before that though. To state it, let us fix some terminology.

If $\iota:W_1\subseteq W_2$  is the inclusion of open monoid subschemes, then $\iota^*:\ope(W_2)\to \ope(W_1)$ is a map given by $L\mapsto L\cap W_1$, $L\in \ope(W_2)$. We also have a similar map
$$\iota^*: \Loc(\mathfrak{Qc}(W_2)) \to \Loc(\mathfrak{Qc}(W_1)),$$
which sends a localising subcategory $\D$ of $\mathfrak{Qc}(W_2)$ to the intersection $\D\cap {\frak Q}{\frak c}(W_1)$. Here, we consider $\mathfrak{Qc}(W_1)$ as a localising subcategory of ${\frak Q}{\frak c}(W_2)$ (see \cite[4.3.2]{points}) and use the fact that the intersection of localising subcategories in a topos is again a localising subcategory \cite[Theorem 6.8 and Example 6.9 iv)]{bk}.

We have a map $\Phi(X):\ope(X)\to \Loc({\mathfrak Q}{\mathfrak c}(X))$, given by $\Phi(U)=\Loc({\mathfrak Q}{\mathfrak c}(U))$. By construction, $\Phi$ is compatible with $\iota$, meaning that the following  diagram commutes
$$\xymatrix{ \ope(W_2)\ar[rr]^{\Phi(W_2)}\ar[d]_{\iota^*} & & \Loc({\mathfrak Q}{\mathfrak c}(W))\ar[d]^{\iota^*} \\
			 \ope(W_1)\ar[rr]_{\Phi(W_1)} & & \Loc({\mathfrak Q}{\mathfrak c}(W_1)).}$$

\begin{Le}\label{290102} Let $X$ be a monoid scheme of finite type. Assume $U_1$ and $U_2$ are open monoid subschems of $X$ and $U_{12}=U_1\cap U_2$, $V=U_1\cup U_2$. Then both
$$\xymatrix{ \ope(V)\ar[r]^{\iota^*_1}\ar[d]^{\iota^*_2} & \ope(U_1)\ar[d]^{\iota_3^*} \\
			 \ope(U_2)\ar[r]_{\iota_4^*} & \ope(U_{12})}$$
and
$$\xymatrix{ \Loc({\mathfrak Q}{\mathfrak c}(V))\ar[r]^{\iota_1^*} \ar[d]^{\iota_2^*} &  \Loc({\mathfrak Q}{\mathfrak c}(U_1))\ar[d]^{\iota^*_3} \\
			 \Loc({\mathfrak Q}{\mathfrak c}(U_2))\ar[r]_{\iota_4^*} & \Loc({\mathfrak Q}{\mathfrak c}(U_{12}))}$$
are pullback diagrams. Here, $\iota_1, \iota_2, \iota_3,\iota_4$ are the appropriate inclusions.
\end{Le}

\begin{proof} Commutativity of the first diagram follows from the definition.

To prove the second part, observe that $\Loc({\mathfrak Q}{\mathfrak c}(V))$ is in one-to-one correspondence with the Lawvere-Tierney topologies on the topos ${\mathfrak Q}{\mathfrak c}(V)$ (see iii) and iv) of Lemma \ref{2510718}). Thus, we can and we will work with Lawvere-Tierney topologies.

Any such topology $j$ on ${\mathfrak Q}{\mathfrak c}(V)$ is a morphism of sheaves $\Omega^V\to \Omega^V$ with appropriate properties. The restriction of $\Omega^V$ on $U_1$ (res. $U_2$, $U_{12}$) is $\Omega^{U_1}$ (res. $\Omega^{U_2}$,$\Omega^{U_{12}}$). The result follows from the fact that for any two morphisms $\Omega^{U_1}\to \Omega^{U_1}$ and $\Omega^{U_2}\to \Omega^{U_2}$, which restrict to the same morphisms $\Omega^{U_{12}}\to \Omega^{U_{12}}$, there exists a unique morphism $\Omega^{V}\to \Omega^{V}$, which restrict to the given morphisms.
\end{proof}

We are now in a position to prove Theorem \ref{maintheorem}, the main theorem of this paper.

\begin{proof}[Proof of Theorem \ref{maintheorem}]

i) and iv) As already sated in the introduction, parts i) and iv) were already proven in \cite{points}.

ii) Observe that by Lemma \ref{290102}, it suffices to restricts ourself to the affine case. This is a consequence of Corollary \ref{26} and Lemma \ref{a290120}.

iii) This follows from Lemma \ref{msz}.

v) Let $p=(p_*,p^*)$ be a point of the topos $\ea$ and $\D$ a localising subcategory of $\ea$. Denote the canonical inclusion $\D\subset \ea$ by $\iota$ and the corresponding localisation functor by $\rho:\ea\to \D$. The pair of morphisms $(\iota,\rho)$ is a geometric morphism  $\D\to \ea$. By our definition, $p\fork \D$ if and only if $p_*:\Sets\to \ea$  factors through $\D$. If this happens, the induced functor $\Sets\to \D$ will be the direct image part of a geometric morphism $\Sets \to \D$ and hence, give rise to a topos point of $\D$.

Equivalently, this happens if and only if $p^*:\ea\to \Sets$ factors through the functor $\rho:\ea\to \D$. That is to say, there exists a functor $p_1:\D\to\Sets$, such that $p^*=p_1\circ \rho$.

Let $\mathfrak{Qc}(X)$ be our topos, $x\in X$ a point and $U$ an open subset of $X$. In this case, the functor $\rho:\mathfrak{Qc}(X)\to\mathfrak{Qc}(U)$ is just the restriction functor. If $x\in U$, then the stalk functor $\Pi(x):\mathfrak{Qc}(X)\to \Sets$ factors through $\mathfrak{Qc}(U)\to \Sets$. Hence, $\Pi(x)\fork \mathfrak{Qc}(U)$. Conversely, if $\Pi(x)\fork \mathfrak{Qc}(U)$, then $\Pi(x)=p_1\circ \rho$, where $p_1$ is the inverse image functor of a topos point of $\mathfrak{Qc}(U)$. As shown in \cite{points}, it corresponds to the stalk at some point $y\in U$. The equality $\Pi(x)=p_1\circ \rho$ shows that for any quasi-coherent sheaf $F$ defined on $X$, we have $F_x=F_y$. Thus $x=y$ and hence $x\in U$.

\end{proof}

\end{document}